\documentclass[11pt,a4]{article}
\usepackage{latexsym}
\usepackage{amsmath}
\usepackage{amssymb}
\usepackage{amsthm}
\usepackage{amscd}
\usepackage{color}
\usepackage[utf8]{inputenc}
\definecolor{dblue}{rgb}{0,0,0.45}
\definecolor{red}{rgb}{0.7,0,0}
\numberwithin{equation}{section}
\numberwithin{equation}{section}
\usepackage{color}

\numberwithin{equation}{section}

\newtheorem{theorem}{Theorem}[section]
\newtheorem{lemma}[theorem]{Lemma}
\newtheorem{corollary}[theorem]{Corollary}
\newtheorem{proposition}[theorem]{Proposition}

\theoremstyle{definition}

\newtheorem{remark}[theorem]{Remark}
\newtheorem{definition}[theorem]{Definition}

\theoremstyle{remark}

\newcommand{\cM}{{\mathcal M}}

\newcommand{\N}{{\mathbb N}}
\newcommand{\R}{{\mathbb R}}

\newcommand{\cG}{{\mathcal G}}

\title{Complex interpolation of vanishing Morrey spaces}
\author{Denny Ivanal Hakim$^{1}$ and Yoshihiro Sawano$^{2,*}$\\
$^1$Department of Mathematics, Bandung Institute of Technology,\\
Jl. Ganesha 10 Bandung 40132, Indonesia.
\\
$^2$Department of Mathematics and Information Sciences, \\ Tokyo Metropolitan University, 1-1 Minami-Osawa,\\
	Hachioji-shi, Tokyo, 192-0397, Japan \\
\footnote{
Yoshihiro Sawano is also affiliated to
Peoples' Friendship University of Russian, Moscow, Russia.
}
Email: $^1$dennyivanalhakim@gmail.com, 
$^{2}$yoshihiro-sawano@celery.ocn.ne.jp	\\}
\begin{document}
\maketitle

\begin{abstract}
We give the description of the first and second complex interpolation of vanishing Morrey spaces, 
introduced 
in \cite{AS, CF}.
In addition, we show that
the diamond subspace (see \cite{HNS}) and one of the function spaces in \cite{AS} are the same.
%%%inroduced->introduced, I also have extended the abstract.
We also give several examples for showing that each of the complex interpolation of these spaces is different. \\
\noindent
{\bf Classification: 42B35, 46B70, 46B26}

\end{abstract}

\section{Introduction}
Let $1\le q\le p<\infty$. The Morrey space 
$\cM^p_q=\cM^p_q(\mathbb{R}^n)$, introduced in 
\cite{Mo}, is defined as the set of all $f\in L^q_{\rm loc}(\mathbb{R}^n)$ for which
\[
\|f\|_{\cM^p_q}
:=
\sup_{r>0}
m(f,p,q;r)<\infty,
\]
where
\[
m(f,p,q;r)
:=
\sup_{x\in \mathbb{R}^n}
|B(x,r)|^{\frac1p}
\left(
\frac{1}{|B(x,r)|}
\int_{B(x,r)} |f(y)|^q \ dy \right)^{\frac1q},\  (r>0).
\]
Note that, for $p=q$, $\cM^p_q$ coincides with the Lebesgue space $L^p$. 
Meanwhile, if $p<q$, then $\cM^p_q$ is strictly larger than $L^p$. For instance, the function $f(x):=|x|^{-\frac{n}{p}}$ belongs to $\cM^p_q$ 
but it is not in $L^p$. 

As a generalization of Lebesgue spaces, one may inquire whether the interpolation of linear operators 
in Morrey spaces also holds. 
The first answer of this question was given by G. Stampacchia in \cite{St}. He proved a partial generalization of the Riesz-Thorin interpolation theorem in Morrey spaces where the domain of the linear operator is assumed to be the Lebesgue spaces. However, when the domain of the linear operator is Morrey spaces,  
there are some counterexamples for the interpolation of linear operator in Morrey spaces (see \cite{BRV, RV}). Although these examples show the lack of interpolation property of Morrey spaces,
there are some recent results about the description of complex interpolation of Morrey spaces. 
The first result in this direction can be found in \cite{CPP}, where the authors proved that if
\begin{align}\label{eq:p0q0}
\theta \in (0,1), 1\le q_0\le p_0<\infty, 1\le q_1\le p_1<\infty
\end{align}
and 
\begin{align}\label{eq:pq}
\frac1p:=
\frac{1-\theta}{p_0}
+
\frac{\theta}{p_1}
\quad 
{\rm and}
\quad
\frac1q:=
\frac{1-\theta}{q_0}
+
\frac{\theta}{q_1},
\end{align}
then
\begin{align}\label{eq:CPP}
[\cM^{p_0}_{q_0}, \cM^{p_1}_{q_1}]_\theta
\subseteq \cM^{p}_{q}.
\end{align}
Here, $[\cdot, \cdot]_\theta$ denotes the first complex interpolation space. 
Assuming the additional assumption 
\begin{align}\label{eq:prop}
\frac{p_0}{q_0}
=
\frac{p_1}{q_1},
\end{align}
Lu et al. \cite{LYY} proved that 
\begin{align}
[\cM^{p_0}_{q_0}, \cM^{p_1}_{q_1}]_\theta
=
\overline{\cM^{p_0}_{q_0} \cap \cM^{p_1}_{q_1}}
^{\cM^p_q}. 
\end{align}
The corresponding result on the second complex interpolation spaces was obtained by Lemari\'e-Rieusset \cite{L14}. 
A generalization  of the results in \cite{LYY, L14} 
in the setting of the generalized Morrey spaces can be seen in \cite{HS, HS2}. 

In addition to complex interpolation of Morrey spaces,  there are several papers on the description of complex interpolation of some closed subspaces of Morrey spaces. 
For instance, Yang et al. \cite{YYZ} proved that 
\begin{align}\label{eq:YYZ}
[\overset{\circ}{\cM}{}^{p_0}_{q_0},
\overset{\circ}{\cM}{}^{p_1}_{q_1}]_\theta
=
\overset{\circ}{\cM}{}^{p}_{q},
\end{align}
where the parameters are given by \eqref{eq:p0q0} and \eqref{eq:pq}
and
$\overset{\circ}{\cM}{}^{p}_{q}$ denotes the closure in $\cM^p_q$ of the set of  smooth functions with compact support. 
Other results on complex interpolation of  closed subspaces of Morrey spaces are considered in \cite{HS2, HNS, HS3, H, YSY}. In particular, the authors in \cite{HS2} consider the space $\overline{\mathcal{M}}{}^p_q:=\overline{L^\infty \cap \mathcal{M}^p_q}^{\mathcal{M}^p_q}$

In this article, we shall investigate complex interpolation of vanishing Morrey spaces. These spaces were introduced in \cite{ AS, CF}. 
Let us recall the their definition as follows. 
\begin{definition}
Let $1\le q\le p<\infty$.
The vanishing Morrey space at the origin $V_0\cM^p_q$ and the vanishing Morrey space at infinity $V_\infty\cM^p_q$ are defined by 
\[
V_0\cM^p_q:=
\{f\in \cM^p_q:
\lim_{r\to 0} m(f,p,q;r)=0\}
\]
and
\[ 
V_\infty\cM^p_q:=
\{f\in \cM^p_q:
\lim_{r\to \infty} m(f,p,q;r)=0\},
\]
respectively. 
The third subspace is the space $V^{(*)}\cM^p_q$ which is defined to be the set of all functions $f\in \cM^p_q$ such that 
\[
\lim_{N\to \infty}
\sup_{x\in \mathbb{R}^n}
\int_{B(x,1)}
|f(y)|^q \chi_{\mathbb{R}^n \setminus B(0,N)}(y) \ dy=0.
\]
\end{definition}

Our main results are the following two theorems.
\begin{theorem}\label{thm:171213-1}
Assume \eqref{eq:p0q0}, \eqref{eq:prop}, and $q_0\neq q_1$. 
Define $p$ and $q$ by \eqref{eq:pq}. 
Then 
\begin{align}\label{eq:171213-2}
[V_0{\mathcal M}^{p_0}_{q_0},V_0{\mathcal M}^{p_1}_{q_1}]_\theta 
&=[{\mathcal M}^{p_0}_{q_0},{\mathcal M}^{p_1}_{q_1}]_\theta
\notag\\
&=
\left\{f\in  \cM^p_q:
\lim_{N\to \infty}\|f-\chi_{\{\frac{1}{N}\le |f|\le N\}}f\|_{\cM^p_q}=0\right\},
\end{align}
\begin{align}\label{eq:171213-3}
[V_\infty{\mathcal M}^{p_0}_{q_0}, V_\infty {\mathcal M}^{p_1}_{q_1}]_\theta 
&=
V_\infty\cM^p_q \cap [{\mathcal M}^{p_0}_{q_0},{\mathcal M}^{p_1}_{q_1}]_\theta
\notag \\
&=
\left\{f\in V_\infty\cM^p_q:
\lim_{N\to \infty}\|f-\chi_{\{\frac{1}{N}\le |f|\le N\}}f\|_{\cM^p_q}=0\right\},
\end{align}
and 
\begin{align}\label{eq:171213-4}
[V^{(*)}{\mathcal M}^{p_0}_{q_0},V^{(*)}{\mathcal M}^{p_1}_{q_1}]_\theta 
%%%(( was duplicated
&=
V^{(*)}\cM^p_q \cap [{\mathcal M}^{p_0}_{q_0},{\mathcal M}^{p_1}_{q_1}]_\theta
\notag \\
&=
\left\{f\in V^{(*)}\cM^p_q:
\lim_{N\to \infty}\|f-\chi_{\{\frac{1}{N}\le |f|\le N\}}f\|_{\cM^p_q}=0\right\}.
\end{align}
\end{theorem}
\begin{theorem}\label{thm:171213-2}
Assume \eqref{eq:p0q0} and \eqref{eq:prop}. 
Define $p$ and $q$ by \eqref{eq:pq}. 
\begin{align}\label{eq:thm:171213-2-1}
[V_0{\mathcal M}^{p_0}_{q_0},V_0{\mathcal M}^{p_1}_{q_1}]^\theta 
=
{\mathcal M}^p_q,
\end{align}
\begin{align}\label{eq:thm:171213-2-2}
[V_\infty{\mathcal M}^{p_0}_{q_0}, V_\infty{\mathcal M}^{p_1}_{q_1}]^\theta 
=
\{f \in {\mathcal M}^p_q \,: 
\chi_{\{a\le |f|\le b\}} f\in V_\infty\cM^p_q 
{\ \rm for \ all } \ 0<a<b<\infty
\},
\end{align}
and 
\begin{align}\label{eq:thm:171213-2-3}
[V^{(*)}{\mathcal M}^{p_0}_{q_0}, V^{(*)}{\mathcal M}^{p_1}_{q_1}]^\theta 
=
\{f \in {\mathcal M}^p_q \,: 
\chi_{\{a\le |f|\le b\}} f\in V^{(*)}\cM^p_q 
{\ \rm for \ all } \ 0<a<b<\infty
\}. 
\end{align}
\end{theorem}
Note that  (\ref{eq:171213-2})
and 
(\ref{eq:thm:171213-2-1}) are immediate once we notice
that $L^\infty \cap {\mathcal M}^p_q \subset V_0{\mathcal M}^p_q$
(see Lemma \ref{lem:171223-1}). In addition to the vanishing Morrey spaces, we discuss the space ${\mathbb M}^p_q$, that is,
the set of all functions $f\in \mathcal{M}^p_q$ for which
$\displaystyle
\lim_{|y| \to 0}f(\cdot+y)=f
$
in the topology of ${\mathcal M}^p_q$.
These spaces were first introduced in \cite{Zo}.
We show that ${\mathbb M}^p_q$ is equal to 
the diamond space $\overset{\diamond}{\mathcal M}{}^p_q$, namely, the closure 
in $\cM^p_q$ of all functions 
$f$ 
such that
$\partial^\alpha f\in {\mathcal M}^p_q$
for all 
$\alpha \in {\mathbb N}_0{}^n$ and $j \in {\mathbb N}$ (see Theorem \ref{thm:171218-1} below). 
As a consequence, the complex interpolation of 
${\mathbb M}^p_q$ follows from the result in \cite{HNS}. Remark that the authors in \cite{AS} also introduced the space $V^{(*)}_{0,\infty}\cM^p_q$ which is defined by 
\[
V^{(*)}_{0,\infty}\cM^p_q:=
V_0\cM^p_q \cap V_\infty\cM^p_q
\cap V^{(*)}\cM^p_q.
\]
Since this space is equal to $\overset{\circ}{\cM}{}^{p}_{q}$ and \eqref{eq:YYZ} holds, we do not consider the complex interpolation of this space. 

The rest of this article is organized as follows. In Section 2 we recall the definition of the complex interpolation method and some previous results about complex interpolation of Morrey spaces and their subspaces. 
We give the proof of Theorems \ref{thm:171213-1} and \ref{thm:171213-2}  in Sections 3 and 4, respectively. 
In Section 5, we show that $\mathbb{M}^p_q$ is equal to $\overset{\diamond}{\mathcal{M}}{}^p_q$. 
Finally, we compare each subspace in Theorems 
\ref{thm:171213-1} and \ref{thm:171213-2} and investigate their relation by giving several examples in Section 6. 

\section{Preliminaries}
\subsection{The complex interpolation method}
Let us recall the definition of complex interpolation method, introduced in \cite{Calderon3}. We follow the presentation in the book \cite{Be}.   Throughout this paper, we define the set $S:=\{z\in \mathbb{C}: 0<{\rm Re}(z)<1\}$ 
and $\overline{S}$ be its closure. First, we recall the following definitions. 
\begin{definition}[Compatible couple]
A couple of Banach spaces $(X_0, X_1)$ is called compatible if there exists a Hausdorff topological vector space $Z$ for which $X_0$ and $X_1$ are continuously embedded into $Z$.
\end{definition}
\begin{definition}[The first complex interpolation functor]
Let $(X_0, X_1)$ be a compatible couple of Banach spaces. The space $\mathcal{F}(X_0, X_1)$ is defined to be the set of all  bounded continuous function $F:\overline{S} \to X_0+X_1$ for which 
\begin{enumerate}
\item 
$F$ is holomorphic in $S$;
\item 
For each $k=0,1$,
the function $t\in \mathbb{R} \mapsto F(k+it) \in X_k$ is bounded and continuous. 
\end{enumerate}
For every $F\in \mathcal{F}(X_0,X_1)$, we define the norm
\[
\|F\|_{\mathcal{F}(X_0, X_1)}
:=
\max_{k=0,1} \sup_{t\in \mathbb{R}}
\|F(k+it)\|_{X_k}.
\]
\end{definition}
\begin{definition}[The first complex interpolation space]
Let $\theta \in (0,1)$. The first complex interpolation of a compatible couple of Banach spaces $(X_0, X_1)$ is defined by 
\[
[X_0, X_1]_\theta 
:=\{F(\theta): F\in \mathcal{F}(X_0, X_1)\}.
\] 
The norm on $[X_0, X_1]_\theta$ is defined by
\[
\|f\|_{[X_0, X_1]_\theta 
}
:=\inf \{\|F\|_{\mathcal{F}(X_0,X_1)}:
f=F(\theta), F\in \mathcal{F}(X_0,X_1) \}. 
\]
\end{definition}
\noindent
We shall use the following density result. 
\begin{lemma}\label{lem:dense}{\rm \cite{Calderon3}}
Let $\theta \in (0,1)$ and given a compatible couple of Banach spaces $(X_0, X_1)$.
Then the space $X_0\cap X_1$ is dense in $[X_0, X_1]_\theta$.
\end{lemma}

We now consider the second complex interpolation method. 
Let $X$ be a Banach space and recall that
the space ${\rm Lip}({\mathbb R},X)$ is defined to be the set of all $X$-valued functions $f$ on $\mathbb{R}$ for which
\[
\|f\|_{{\rm Lip}({\mathbb R},X)}
:=
\sup_{-\infty<t<s<\infty}
\frac{\|f(s)-f(r)\|_X}{|s-t|}
\]
is finite. The definition of the second complex interpolation space is given as follows.
\begin{definition}[The second complex interpolation functor]
Let $(X_0, X_1)$ be a compatible couple of Banach spaces. 
The space $\mathcal{G}(X_0, X_1)$ is the set of all continuous functions $G:\overline{S} \to X_0+X_1$ for which 
\begin{enumerate}
\item 
$G|_{S}$ is holomorphic;
\item 
$\displaystyle \sup\limits_{z\in \overline{S}} \frac{\left\|G(z)\right\|_{X_0+X_1}}{1+|z|} <\infty$;
\item 
For each $k=0,1$, the function $t\in \mathbb{R} \mapsto G(k+it) \in X_k$ belongs to ${\rm Lip}({\mathbb R}, X_k)$.
\end{enumerate}
For every $G\in \mathcal{G}(X_0, X_1)$, we define 
\[
\|G\|_{\mathcal{G}(X_0, X_1)}
:=
\max_{k=0,1}
\sup_{t\in \mathbb{R}}
\|G(k+i\cdot)\|_{{\rm Lip}(\R, X_k)}.
\]
\end{definition}
\begin{definition}[The second complex interpolation space]
Let $\theta \in (0,1)$ and $(X_0, X_1)$ be a compatible couple of Banach spaces.  The second complex interpolation space $[X_0, X_1]^\theta$
is defined by 
\[
[X_0, X_1]^\theta
:=
\{G'(\theta): G\in \mathcal{G}(X_0, X_1)\}. 
\]
The space $[X_0, X_1]^\theta$ is equipped with the norm
\[
\|f\|_{[X_0, X_1]^\theta}
:=
\inf\{\|G\|_{\mathcal{G}(X_0, X_1)}:
f=G'(\theta), G\in \mathcal{G}(X_0, X_1)
\}. 
\]
\end{definition}
We shall utilize the following relation between the first and second complex interpolation method. 
\begin{lemma}\label{lem:160417-3}
	{\rm \cite[Lemma 2.4]{HS2}}
Let  $(X_0, X_1)$ be a compatible couple of Banach spaces and
let $G\in \cG(X_0,X_1)$ be fixed.  
For every $z\in \overline{S}$,
and $k\in \N$, set
\begin{align}\label{eq:160417-3}
H_k(z):=\frac{G(z+2^{-k}i) -G(z)}{2^{-k}i}.
\end{align}
Then, $H_k(\theta)\in [X_0,X_1]_\theta$, for every $\theta \in (0,1)$.
\end{lemma}

\subsection{Previous results on complex interpolation of Morrey spaces}

First, let us recall the results on the second complex interpolation method of Morrey spaces. 
\begin{proposition}\label{prop:Lemarie-HS}{\rm \cite{HS, L14}}
Keep the same assumption as in Theorem 
\ref{thm:171213-2}. 
Let $f\in \mathcal{M}^p_q$. 
Define the functions $F$ and $G$ on $\overline{S}$ by
\begin{align}\label{eq:171220-20}
F(z):={\rm sgn}(f) |f|^{p\left(\frac{1-z}{p_0}+\frac{z}{p_1}\right)}, \ 
(z\in \overline{S})
\end{align}
and 
\begin{align}\label{eq:171220-210}
G(z)
:=
(z-\theta)
\int_0^1 F(\theta+(z-\theta)t)\ dt,\  (z\in \overline{S}).
\end{align}
Then, for every $z\in \overline{S}$, we have
\begin{align}\label{eq:171220-27}
|G(z)|\le 
(1+|z|)\left(|f|^{\frac{p}{p_0}}+|f|^{\frac{p}{p_1}}\right).
\end{align}
Moreover, 
$G\in \mathcal{G}(\mathcal{M}^{p_0}_{q_0}, \mathcal{M}^{p_1}_{q_1})$. 
\end{proposition}
\begin{theorem}
\label{thm:Lemarie}{\rm \cite{L14}}
Keep the same assumption as in Theorem 
\ref{thm:171213-2}. Then $$[\mathcal{M}^{p_0}_{q_0}, \mathcal{M}^{p_1}_{q_1}]^\theta =\mathcal{M}^p_q.$$
\end{theorem}
The description of complex interpolation of some closed subspaces of Morrey spaces is given as follows. 
\begin{theorem}\label{thm:HS2}{\rm \cite{HS2}}
Keep the same assumption as in Theorem 
\ref{thm:171213-1}. 
Then 
\[
[\overline{\mathcal{M}}{}^{p_0}_{q_0}, 
\overline{\mathcal{M}}{}^{p_1}_{q_1}]_\theta
=
[\mathcal{M}^{p_0}_{q_0}, \mathcal{M}^{p_1}_{q_1}]_\theta
=
\{f\in \mathcal{M}^p_q: \lim_{N\to \infty}
\|f-\chi_{\{1/N\le |f|\le N\}}f\|_{\mathcal{M}^p_q}=0\}.
\]
\end{theorem}

\begin{theorem}\label{thm:HS2-2}{\rm \cite{HS2}}
Keep the same assumption as in Theorem 
\ref{thm:171213-2}. Then 
\[
[\overline{\mathcal{M}}{}^{p_0}_{q_0}, 
\overline{\mathcal{M}}{}^{p_1}_{q_1}]^\theta
=
{\mathcal{M}}{}^p_q.
%%%This is not a bar space.
\]
\end{theorem}

We now recall the complex interpolation results of the diamond spaces in \cite{HNS}. 
To state these results, we recall the following notation. 
\begin{definition}\label{d-1}
Let $\psi\in C^\infty_{\rm c}(\mathbb{R}^n)$
satisfy
$\chi_{Q(4)}\leq \psi\leq \chi_{Q(8)}$,
where $Q(r):=[-r,r]^n$.
Set $\varphi_0:=\psi$ and for $j\in\N$, define
\begin{equation*}\label{eq:151113-21}
\varphi_j:=\psi(2^{-j}\cdot)-\psi(2^{-j+1}\cdot).
\end{equation*}
We also define $\varphi_j(D)f:=\mathcal{F}^{-1}(\varphi_j \cdot \mathcal{F}f)$, where $\mathcal{F}$ and $\mathcal{F}^{-1}$
denote the Fourier transform and its inverse. 
For $a\in (0,1)$, $J\in \mathbb{N}$, and a measurable function $f$, we define
\[S(f):=\left(\sum_{j=0}^\infty|\varphi_j(D)f|^2\right)^{\frac12}
\ 
{\rm and}
\ 
S(f;a,J):=
\chi_{\{a\le S(f)\le a^{-1}\}}
\left(\sum_{j=J}^\infty|\varphi_j(D)f|^2\right)^{\frac12}.
\]
\end{definition}
Using the notation in Definition \ref{d-1}, let us state the description of complex interpolation of diamond spaces.
\begin{theorem}
{\rm \cite[Theorem 1.4]{HNS}}
%%%I have added the pinpoint information.
\label{thm:HNS}
Let $\theta \in (0,1)$, $1<q_0\le p_0<\infty$,
and $1<q_1\le p_1<\infty$. 
Assume the condition \eqref{eq:prop}.
Define $p$ and $q$ by \eqref{eq:pq}.
Then 
\begin{align*}
[\overset{\diamond}{{\mathcal M}}{}^{p_0}_{q_0},
\overset{\diamond}{{\mathcal M}}{}^{p_1}_{q_1}]_\theta
=
\{f\in
\overset{\diamond}{{\mathcal M}}{}^{p}_{q}
: \lim_{N\to \infty}
\|f-\chi_{\{1/N\le |f|\le N\}}f\|_{\cM^p_q}=0
\}
\end{align*}
and
\begin{align*}
[\overset{\diamond}{{\mathcal M}}{}^{p_0}_{q_0},
\overset{\diamond}{{\mathcal M}}{}^{p_1}_{q_1}]^\theta
=
\bigcap_{0<a<1}
\left\{
f \in {\mathcal M}^p_q
\,:\,
\lim_{J \to \infty}
\|S(f;a,J)\|_{{\mathcal M}^p_q}
=0
\right\}.
\end{align*}
\end{theorem}
\section{The first complex interpolation of vanishing Morrey spaces}
\begin{lemma}\label{lem:171223-1}
Let $1\le q\le p<\infty$. Then, $\overline{\mathcal{M}}{}^{p}_q \subseteq V_0\mathcal{M}^p_q$. 
%In particular, if  $f\in \mathcal{M}^p_q$, 
%then for every $0<a<b<\infty$, we have
%\begin{align}\label{eq:171223-1}
%\chi_{\{a\le |f|\le b\}} f\in V_0\mathcal{M}^p_q.
%\end{align}
\end{lemma}
\begin{proof}
Let $g\in L^\infty \cap \mathcal{M}^p_q$. Then, for every $r>0$, we have
\[
m(g, p, q; r)\lesssim \|g\|_{L^\infty} r^{\frac{n}{p}},
\]
so $\lim\limits_{r\to 0^+} m(g, p, q; r)=0$. Hence, $g\in V_0\mathcal{M}^p_q$. 
Thus, $L^\infty \cap \mathcal{M}^p_q \subseteq V_0\mathcal{M}^p_q$. 
Since $V_0\mathcal{M}^p_q$ is a closed subspace of $\mathcal{M}^p_q$, we conclude that $\overline{\mathcal{M}}{}^{p}_q \subseteq V_0\mathcal{M}^p_q$. 
\end{proof}

\begin{lemma}\label{lem:171227-1}
Let $1\le q< p<\infty$. Then $L^\infty_{\rm c} \subseteq V_\infty\mathcal{M}^p_q \cap V^{(*)}\mathcal{M}^p_q$.
\end{lemma}
\begin{proof}
Let $f\in L^\infty_{\rm c}$. Then, for every $r>0$, we have
\[
m(f, p, q;r)
\lesssim 
r^{\frac{n}{p}-\frac{n}{q}}
\|f\|_{L^q}.
\]
Consequently, $\lim\limits_{r\to \infty}m(f, p, q;r)=0$. Therefore, $f\in V_\infty \mathcal{M}^p_q$. We now show that 
\begin{align}\label{eq:171227-4}
f\in V^{(*)} \mathcal{M}^p_q.
\end{align}
For every $N\in \mathbb{N}$, we  have 
\[
\sup_{x\in \mathbb{R}^n}
\int_{B(x,1)} |f(y)|^q \chi_{\mathbb{R}^n \setminus B(0,N)}(y) \ dy\le \|f\chi_{\mathbb{R}^n \setminus B(0,N)}\|_{L^q}^q.
\]
Therefore, 
since the right-hand side is zero for large $N$
%I simplified
%by virtue of the dominated convergence theorem, 
we have 
\[
\lim_{N\to \infty}
\sup_{x\in \mathbb{R}^n}
\int_{B(x,1)} |f(y)|^q \chi_{\mathbb{R}^n \setminus B(0,N)}(y) \ dy=0,
\]
which implies \eqref{eq:171227-4}. 
\end{proof}

%{\color{red}
%\begin{lemma}\label{lem:171223-2}
%Let $1\le q\le p<\infty$. If $f\in \mathcal{M}^p_q$ satisfies 
%\begin{align}\label{eq:171223-2}
%\lim_{N\to \infty}\|f-\chi_{\{1/N\le |f|\le N\}}f\|_{\mathcal{M}^p_q}=0,
%\end{align}
%then $f\in V_0\mathcal{M}^p_q$.
%\end{lemma}
%\begin{proof}
%For every $N\in \mathbb{N}$, define 
%$f_N:=\chi_{\{1/N\le |f|\le N\}}f$. Then, for every $r>0$, we have
%\[
%m(f, p, q; r)
%\lesssim 
%Nr^{n/p}+
%\|f-f_N\|_{\mathcal{M}^p_q}.
%\]
%Taking $r\to 0^+$, we get
%\begin{align}\label{eq:171223-3}
%\limsup_{r\to 0^+}
%m(f, p, q; r)
%\lesssim 
%\|f-f_N\|_{\mathcal{M}^p_q}.
%\end{align}
%Combining \eqref{eq:171223-2} and \eqref{eq:171223-3}, we have
%$\lim\limits_{r\to 0^+}
%m(f, p, q; r)=0$, so $f\in V_0\mathcal{M}^p_q$.
%\end{proof}
%}

\begin{lemma}\label{lem:171218-1}
Let $\theta \in (0,1)$, $1\le q_0\le p_0<\infty$, 
and $1\le q_1\le p_1<\infty$. 
Define $p$ and $q$ by 
\[
\frac1p:=\frac{1-\theta}{p_0}+\frac{\theta}{p_1}
\quad 
{\rm and}
\quad
\frac1q:=\frac{1-\theta}{q_0}+\frac{\theta}{q_1}.
\]
Then we have the following inclusions:
\[
V_\infty\cM^{p_0}_{q_0}
\cap 
V_\infty\cM^{p_1}_{q_1}
\subseteq 
V_\infty\cM^{p}_{q}, 
\quad 
{\rm and}
\quad 
V^{(*)}\cM^{p_0}_{q_0}
\cap 
V^{(*)}\cM^{p_1}_{q_1}
\subseteq 
V^{(*)}\cM^{p}_{q}.
\]
\end{lemma}
\begin{proof}
We only prove the first inclusion. 
The proof of 
another
%typo anothe->another
inclusion is similar. Let $f\in V_\infty\cM^{p_0}_{q_0}
\cap 
V_\infty\cM^{p_1}_{q_1}$. Then 
\begin{align}\label{eq:171213-1}
\lim_{r\to \infty}
m(f,p_0,q_0;r)
=0
\quad 
{\rm and}
\quad
\lim_{r\to \infty}
m(f,p_1,q_1;r)
=0.
\end{align}
By H\"older's inequality, for every $r>0$, we have
\[
m(f,p,q;r)
\le 
m(f,p_0,q_0;r)^{1-\theta}
m(f,p_1,q_1;r)^{\theta}.
\]
Combining this inequality and \eqref{eq:171213-1}, we get $\lim\limits_{r\to \infty}m(f,p,q;r)=0$, so $f\in V_\infty\cM^p_q$, as desired. 
\end{proof}

\begin{proposition}\label{pr:171218-1}
Let $f\in V_\infty\cM^p_q$ be such that 
$\displaystyle
f=\lim_{N \to \infty}\chi_{\{1/N\le |f|\le N\}}f
$
in ${\mathcal M}^p_q$.
For a fixed $N\in \mathbb{N}$, define
\begin{align}\label{eq:FN}
F_N(z)={\rm sgn}(f)|f|^{p\frac{1-z}{p_0}+p\frac{z}{p_1}}\chi_{\{1/N\le |f| \le N\}}  \quad
(z\in \overline{S}).
\end{align}
Then $F_N(\theta)\in [V_\infty\mathcal{M}^{p_0}_{q_0}, V_\infty\mathcal{M}^{p_1}_{q_1}]_\theta$.
\end{proposition}
\begin{proof}
%{\color{red} Note that, by virtue of Lemma \ref{lem:171223-2}, we have $f\in V_0\mathcal{M}^p_q$.} 
Observe that \eqref{eq:pq} and \eqref{eq:prop} imply 
\begin{align}\label{eq:171220-10}
\frac{p_0}{q_0}
=
\frac{p_1}{q_1}
=
\frac{p}{q}.
\end{align}
Without loss of generality, assume that $p_0>p_1$. 
Define $F_{N,0}(z):=\chi_{\{|f|\le 1\}}F_N(z)$ and 
$F_{N,1}(z):=F_N(z)-F_{N,0}(z)$. Since
\[
|F_{N,0}(z)|=
\chi_{\{\frac1N\le |f|\le 1\}}
|f|^{\frac{p}{p_0}}
|f|^{\left(\frac{p}{p_1}-\frac{p}{p_0}\right){\rm Re}(z)}
\le |f|^{\frac{p}{p_0}},
\]
by using \eqref{eq:171220-10}, 
we have
\begin{align}\label{eq:171220-7}
m(F_{N,0}(z), p_0, q_0;r)
\le 
m\left(|f|^{\frac{p}{p_0}}, p_0, q_0;r\right)
=
m(f, p, q;r)^{\frac{p}{p_0}}.
\end{align}
Taking $r\to 0^{+}$ and using the fact that 
$f\in V_\infty\mathcal{M}^p_q$, we have $F_{N,0}(z)\in V_\infty\mathcal{M}^{p_0}_{q_0}$. Moreover, by \eqref{eq:171220-7}, we also have
\begin{align}\label{eq:171220-8}
\|F_{N,0}(z)\|_{V_\infty\mathcal{M}^{p_0}_{q_0}}
\le 
\|f\|_{\mathcal{M}^p_q}^{\frac{p}{p_0}}.
\end{align}
By a similar argument, we have $F_{N,1}(z)\in V_\infty\mathcal{M}^{p_1}_{q_1}$ and 
\begin{align}\label{eq:171220-9}
\|F_{N,1}(z)\|_{V_\infty\mathcal{M}^{p_1}_{q_1}}
\le 
\|f\|_{\mathcal{M}^p_q}^{\frac{p}{p_1}}.
\end{align}
Combining \eqref{eq:171220-8} and \eqref{eq:171220-9}, we have $F_{N}(z)\in V_\infty\mathcal{M}^{p_0}_{q_0}+V_\infty\mathcal{M}^{p_1}_{q_1}$ and
\begin{align}\label{eq:171220-11}
\sup_{z\in \overline{S}} 
\|F_N(z)\|_{V_\infty\mathcal{M}^{p_0}_{q_0}+V_\infty\mathcal{M}^{p_1}_{q_1}}
\le 
\|f\|_{\mathcal{M}^p_q}^{\frac{p}{p_0}}
+\|f\|_{\mathcal{M}^p_q}^{\frac{p}{p_1}}
<\infty.
\end{align}
We now show the continuity of $F_N$. Let $z\in \overline{S}$ and $h\in \overline{S}$ be such that $z+h\in \overline{S}$. Since
\begin{align*}
|F_{N,0}(z+h)-F_{N,0}(z)|
&=
\left|
|f|^{h\left(\frac{p}{p_1}-\frac{p}{p_0}\right)}
-1
\right|
|F_{N,0}(z)|
\\
&\le 
\left(
e^{|h|\left(\frac{p}{p_1}-\frac{p}{p_0}\right)|\log |f||}-1
\right)
|F_{N,0}(z)|
\\
&\le 
\left(
e^{|h|\left(\frac{p}{p_1}-\frac{p}{p_0}\right)\log N}-1
\right)
|F_{N,0}(z)|,
\end{align*}
by using \eqref{eq:171220-8}, we have
\begin{align}\label{eq:171220-12}
\|F_{N,0}(z+h)-F_{N,0}(z)\|_{V_\infty\mathcal{M}^{p_0}_{q_0}}
\le 
\left(
e^{|h|\left(\frac{p}{p_1}-\frac{p}{p_0}\right)\log N}-1
\right)
\|f\|_{\mathcal{M}^p_q}^{\frac{p}{p_0}}.
\end{align}
Similarly, 
\begin{align}\label{eq:171220-13}
\|F_{N,1}(z+h)-F_{N,1}(z)\|_{V_\infty\mathcal{M}^{p_1}_{q_1}}
\le 
\left(
e^{|h|\left(\frac{p}{p_1}-\frac{p}{p_0}\right)\log N}-1
\right)
\|f\|_{\mathcal{M}^p_q}^{\frac{p}{p_1}}.
\end{align}
Combining \eqref{eq:171220-12} and \eqref{eq:171220-13}, we get
\begin{align*}
\|F_{N}(z+h)-F_{N}(z)\|_{V_\infty\mathcal{M}^{p_0}_{q_0}+V_\infty\mathcal{M}^{p_1}_{q_1}}
\le 
\left(
e^{|h|\left(\frac{p}{p_1}-\frac{p}{p_0}\right)\log N}-1
\right)
\left(\|f\|_{\mathcal{M}^p_q}^{\frac{p}{p_0}}
+
\|f\|_{\mathcal{M}^p_q}^{\frac{p}{p_1}}
\right). 
\end{align*}
This implies
\[
\lim_{h\to 0}
\|F_{N}(z+h)-F_{N}(z)\|_{V_\infty\mathcal{M}^{p_0}_{q_0}+V_\infty\mathcal{M}^{p_1}_{q_1}}
=0. 
\]
Hence, $F_N$ is continuous on $\overline{S}$.
The proof of holomorphicity of $F_N$ in $S$ goes as follows. 
For every $z\in \overline{S}$, define 
\[
F_{N,0}'(z):=F_{N,0}(z)\left(\frac{p}{p_1}-\frac{p}{p_0}\right) \log |f|, 
\ 
F_{N,1}'(z):=F_{N,1}(z)\left(\frac{p}{p_1}-\frac{p}{p_0}\right) \log |f|,
\  
\]
and $F_{N}'(z):=F_{N,0}'(z)+F_{N,1}'(z)$. 
As a consequence of \eqref{eq:171220-8}
and \eqref{eq:171220-9}, we have
\begin{align*}
\|F_N'(z)\|_{V_\infty\mathcal{M}^{p_0}_{q_0}
+V_\infty\mathcal{M}^{p_1}_{q_1}}
&\le 
\|F_{N,0}'(z)\|_{V_\infty\mathcal{M}^{p_0}_{q_0}}
+
\|F_{N,1}'(z)\|_{V_\infty\mathcal{M}^{p_1}_{q_1}}
\\
&\le 
\left(\frac{p}{p_1}-\frac{p}{p_0}\right)
(\log N) \left(\|f\|_{\mathcal{M}^p_q}^{\frac{p}{p_0}}
+\|f\|_{\mathcal{M}^p_q}^{\frac{p}{p_1}}
\right),
\end{align*}
so $F_N'(z)\in V_\infty\mathcal{M}^{p_0}_{q_0}
+V_\infty\mathcal{M}^{p_1}_{q_1}$. Now, let $z\in S$ and $h\in \mathbb{C}\setminus \{0\}$ be such that $z+h\in \overline{S}$. Then
\begin{align}\label{eq:171221-1}
&\left|
\frac{F_{N,0}(z+h)-F_{N,0}(z)}{h}
-F_{N,0}'(z)
\right|
\nonumber
\\
&\qquad\le 
|F_{N,0}(z)|
\left(\frac{p}{p_1}-\frac{p}{p_0}
\right)
|\log |f|| 
\left( e^{|h|\left(\frac{p}{p_1}-\frac{p}{p_0}\right) |\log |f||}
-1\right)
\nonumber
\\
&\qquad\le 
|F_{N,0}(z)|
\left(\frac{p}{p_1}-\frac{p}{p_0}
\right)
(\log N)
\left( e^{|h|\left(\frac{p}{p_1}-\frac{p}{p_0}\right) \log N}
-1\right).
\end{align}
Combining \eqref{eq:171220-8} and \eqref{eq:171221-1}, we get
\begin{align}\label{eq:171221-2}
\left\|
\frac{F_{N,0}(z+h)-F_{N,0}(z)}{h}
-F_{N,0}'(z)
\right\|_{V_\infty\mathcal{M}^{p_0}_{q_0}}
\lesssim
\left( e^{|h|\left(\frac{p}{p_1}-\frac{p}{p_0}\right) \log N}
-1\right)\|f\|_{\mathcal{M}^p_q}^{\frac{p}{p_0}}.
\end{align}
Similarly, 
\begin{align}\label{eq:171221-3}
\left\|
\frac{F_{N,1}(z+h)-F_{N,1}(z)}{h}
-F_{N,1}'(z)
\right\|_{V_\infty\mathcal{M}^{p_1}_{q_1}}
\lesssim
\left( e^{|h|\left(\frac{p}{p_1}-\frac{p}{p_0}\right) \log N}
-1\right)\|f\|_{\mathcal{M}^p_q}^{\frac{p}{p_1}}.
\end{align}
Here the implicit constants
in (\ref{eq:171221-1}) and (\ref{eq:171221-2})
can depend on $N$.
%%%This is just a matter of my taste. Allow me to do so.
Hence, it follows from \eqref{eq:171221-2} and 
\eqref{eq:171221-3} that 
\[
\lim_{h\to 0}
\left\|
\frac{F_{N}(z+h)-F_{N}(z)}{h}
-F_{N}'(z)
\right\|_{V_\infty\mathcal{M}^{p_0}_{q_0}+V_\infty\mathcal{M}^{p_1}_{q_1}}
=0,
\]
so $F_N$ is holomorphic in $S$. 
Finally, we show the boundedness and continuity of 
the function $t\in \mathbb{R} \mapsto F_N(k+it)\in V_\infty\mathcal{M}^{p_k}_{q_k}$ for each $k\in \{0,1\}$. Note that, by using \eqref{eq:171220-10}, we have
\[
m(F_N(k+it), p_k, q_k;r)
\le 
m(|f|^{\frac{p}{p_k}}, p_k, q_k;r)
=
m(f, p_k, q_k;r)^{\frac{p}{p_k}},
\]
so $F_{N}(k+it)\in V_\infty\mathcal{M}^{p_k}_{q_k}$
and 
\begin{align}\label{eq:171221-4}
\sup_{t\in \mathbb{R}}
\|F_N(k+it)\|_{V_\infty\mathcal{M}^{p_k}_{q_k}}
\le 
\|f\|_{\mathcal{M}^p_q}^{\frac{p}{p_k}}. 
\end{align}
Hence, $t\in \mathbb{R} \mapsto F_N(k+it)\in V_\infty\mathcal{M}^{p_k}_{q_k}$ is bounded. 
Let $t_0\in \mathbb{R}$ be fixed. Then, by \eqref{eq:171221-4}, for every $t\in \mathbb{R}$, we have
\begin{align*}
\|F_N(k+it)-
F_N(k+it_0)\|_{V_\infty\mathcal{M}^{p_k}_{q_k}}
&=
\left\|F_N(k+it_0)\left(|f|^{i\left(\frac{p}{p_0}-\frac{p}{p_1}\right)(t-t_0)}-1\right)
\right\|_{V_\infty\mathcal{M}^{p_k}_{q_k}}
\\
&\le 
\|f\|_{\mathcal{M}^p_q}^{\frac{p}{p_k}}
\left(
e^{\left(\frac{p}{p_1}-\frac{p}{p_0}\right) (\log N)|t_1-t_2|}
-1\right),
\end{align*}
so 
\[
\lim_{t\to t_0}
\|F_N(k+it)-
F_N(k+it_0)\|_{V_\infty\mathcal{M}^{p_k}_{q_k}}
=0.
\] 
This shows that $t\in \mathbb{R} \mapsto F_N(k+it)\in V_\infty\mathcal{M}^{p_k}_{q_k}$ is continuous.
Hence, we have shown that $F_N\in \mathcal{F}(V_\infty\mathcal{M}^{p_0}_{q_0}, V_\infty\mathcal{M}^{p_1}_{q_1})$. 
Thus, $F_N(\theta)\in [V_\infty\mathcal{M}^{p_0}_{q_0}, V_\infty\mathcal{M}^{p_1}_{q_1}]_\theta$.
\end{proof}
\begin{remark}
The similar result is also valid when $V_\infty$ is replaced by $V^{(*)}$.
\end{remark}
We are now ready to prove Theorem \ref{thm:171213-1}.
\begin{proof}[Proof of Theorem \ref{thm:171213-1}]
According to Lemma \ref{lem:171223-1}, we have
$\overline{\mathcal{M}}{}^{p_0}_{q_0} \subseteq V_0\mathcal{M}^{p_0}_{q_0} \subseteq \mathcal{M}^{p_0}_{q_0}$
and 
$\overline{\mathcal{M}}{}^{p_1}_{q_1} \subseteq V_0\mathcal{M}^{p_1}_{q_1} \subseteq \mathcal{M}^{p_1}_{q_1}$. 
Consequently, by virtue of Theorem  \ref{thm:HS2-2}, we have \eqref{eq:171213-2}. 

Next, we only prove  \eqref{eq:171213-3} because the proofs of \eqref{eq:171213-3} and \eqref{eq:171213-4} are similar. 
Let $f\in [V_\infty \mathcal{M}^{p_0}_{q_0}, V_\infty \mathcal{M}^{p_1}_{q_1}]_\theta$. 
Then, by virtue of Lemma \ref{lem:dense}, we can choose $\{f_j\}_{j=1}^\infty \subseteq V_\infty \mathcal{M}^{p_0}_{q_0} \cap V_\infty \mathcal{M}^{p_1}_{q_1}$ such that 
\begin{align}\label{eq:171219-1}
\lim_{j\to \infty}
\|f-f_j\|_{[V_\infty \mathcal{M}^{p_0}_{q_0}, V_\infty \mathcal{M}^{p_1}_{q_1}]_\theta}
=0.
\end{align}
According to Lemma \ref{lem:171218-1}, we have $\{f_j\}_{j=1}^\infty \subseteq V_\infty\mathcal{M}^p_q$.
Combining $[V_\infty\mathcal{M}^{p_0}_{q_0}, V_\infty\mathcal{M}^{p_1}_{q_1}]_\theta \subseteq [\mathcal{M}^{p_0}_{q_0}, \mathcal{M}^{p_1}_{q_1}]_\theta \subseteq \mathcal{M}^p_q$ and 
\eqref{eq:171219-1}, we get
\begin{align}\label{eq:171219-2}
\lim_{j\to \infty}
\|f-f_j\|_{\mathcal{M}^p_q}
=0,
\end{align}
so $f\in V_\infty\mathcal{M}^p_q$. 
Consequently,
\begin{align*}
[V_\infty{\mathcal M}^{p_0}_{q_0}, V_\infty{\mathcal M}^{p_1}_{q_1}]_\theta 
\subseteq
V_\infty \cM^p_q 
\cap 
[{\mathcal M}^{p_0}_{q_0}, {\mathcal M}^{p_1}_{q_1}]_\theta 
\end{align*}
Conversely, let 
$V_\infty \cM^p_q 
\cap 
[{\mathcal M}^{p_0}_{q_0}, {\mathcal M}^{p_1}_{q_1}]_\theta$ . 
Then,  by virtue of Theorem \ref{thm:HS2}, we have
\begin{align}\label{eq:171220-1}
\lim_{N\to \infty}\|f-\chi_{\{1/N\le |f|\le N\}}f\|_{\mathcal{M}^p_q}=0.
\end{align}
Define $F_N(z)$ by \eqref{eq:FN}. Then, by virtue of  Proposition
\ref{pr:171218-1}, we have
$$F_N(\theta)\in [V_\infty\mathcal{M}^{p_0}_{q_0}, V_\infty\mathcal{M}^{p_1}_{q_1}]_\theta.$$ Moreover, for every $M, N \in \mathbb{N}$ with $M>N$, we have
\begin{align*}
\|F_M(\theta)-F_N(\theta)\|_{[V_\infty\mathcal{M}^{p_0}_{q_0}, V_\infty\mathcal{M}^{p_1}_{q_1}]_\theta}
\le 
\max_{k=0,1}
\|f\chi_{\{|f|<\frac1N\}\cup\{|f|>N\}}\|_{\mathcal{M}^p_q}^{\frac{p}{p_k}},
\end{align*}
so by \eqref{eq:171220-1}, we see that 
$\{F_N(\theta)\}_{N=1}^\infty$ is a Cauchy sequence in $[V_\infty\mathcal{M}^{p_0}_{q_0}, V_\infty\mathcal{M}^{p_1}_{q_1}]_\theta$. 
Consequently, there exists $g\in [V_\infty\mathcal{M}^{p_0}_{q_0}, V_\infty\mathcal{M}^{p_1}_{q_1}]_\theta$ such that 
\begin{align}\label{eq:171220-2}
\lim_{N\to \infty}
\|F_N(\theta)-g\|_{[V_\infty\mathcal{M}^{p_0}_{q_0}, V_\infty\mathcal{M}^{p_1}_{q_1}]_\theta}
=0. 
\end{align}
By using $[V_\infty\mathcal{M}^{p_0}_{q_0}, V_\infty\mathcal{M}^{p_1}_{q_1}]_\theta \subseteq [\mathcal{M}^{p_0}_{q_0}, \mathcal{M}^{p_1}_{q_1}]_\theta \subseteq \mathcal{M}^p_q$ again, we see that \eqref{eq:171220-2} implies
\begin{align}\label{eq:171220-3}
\lim_{N\to \infty}
\|F_N(\theta)-g\|_{\mathcal{M}^p_q}
=0. 
\end{align}
Since $f-F_N(\theta)=f-\chi_{\{1/N\le |f|\le N\}}f$, we may combine \eqref{eq:171220-1} and \eqref{eq:171220-3} to obtain $f=g$, so $f\in [V_\infty \mathcal{M}^{p_0}_{q_0}, V_\infty\mathcal{M}^{p_1}_{q_1}]_\theta$.  
This completes the proof of Theorem \ref{thm:171213-1}.
\end{proof}

\section{The second complex interpolation of vanishing Morrey spaces}
First, we show that the vanishing Morrey
spaces $V_\infty\mathcal{M}^p_q$ is a Banach lattice on $\mathbb{R}^n$.
\begin{lemma}\label{lem:171220-1}
Let $1\le q\le p<\infty$. If $f\in V_\infty\mathcal{M}^p_q$ and $|g|\le |f|$, then
$g\in V_\infty\mathcal{M}^p_q$. 
\end{lemma}
\begin{proof}
The assertion follows immediately from the inequality 
\[
m(g, p, q; r)\le 
m(f, p, q;r),
\]
for every $r>0$.
\end{proof}
\begin{remark}\label{rem:171220-1}
By a similar argument, we also can show that 
$V_0\mathcal{M}^p_q$ and $V^{(*)}\mathcal{M}^p_q$ are Banach lattices on $\mathbb{R}^n$.  
\end{remark}

We now prove the following inclusion result, whose proof is similar to that of \cite[Lemma 8]{HS}. 
\begin{lemma}\label{lem:171220-2}
Keep the same assumption as in Theorem \ref{thm:171213-2}. Then 
\[
\mathcal{M}^p_q
\cap 
\overline{V_\infty\mathcal{M}^p_q}^{\mathcal{M}^{p_0}_{q_0}+\mathcal{M}^{p_1}_{q_1}}
\subseteq 
\bigcap_{0<a<b<\infty}
\{f\in \mathcal{M}^p_q:
\chi_{\{a\le |f|\le b\}}f \in V_\infty\mathcal{M}^p_q\}.
\]
\end{lemma}
\begin{proof}
We may assume that $q_0>q_1$.
Then $q_0>q>q_1$.  
Let $f\in \mathcal{M}^p_q
\cap 
\overline{V_\infty\mathcal{M}^p_q}^{\mathcal{M}^{p_0}_{q_0}+\mathcal{M}^{p_1}_{q_1}}
$. We shall show that, for every $0<a<b<\infty$ 
\begin{align}\label{eq:171221-7}
\chi_{\{a\le |f|\le b\}}f \in V_\infty\mathcal{M}^p_q.
\end{align}
In view of Lemma \ref{lem:171220-1}, we can prove \eqref{eq:171221-7} by showing that
\begin{align}\label{eq:171221-8}
\chi_{\{a\le |f|\le b\}}\Theta(|f|) \in V_\infty\mathcal{M}^p_q,
\end{align}
where $\Theta:[0,\infty) \to [0,\infty)$ is defined by 
\begin{align*}
\Theta(t):=
\begin{cases}
0, &\quad 0\le t< \frac{a}{2} \  {\rm or}\ t>2b,
\\
2t-a, &\quad \frac{a}{2}<t\le a,
\\
a, &\quad a\le t\le b, 
\\
-\frac{a}{b}t+2a, &\quad b<t\le 2b.
\end{cases}
\end{align*}
Since $f\in 
\overline{V_\infty\mathcal{M}^p_q}^{\mathcal{M}^{p_0}_{q_0}+\mathcal{M}^{p_1}_{q_1}}$, 
we can choose 
$\{f_j\}_{j=1}^\infty \subseteq V_\infty\mathcal{M}^p_q$, 
$\{g_j\}_{j=1}^\infty \subseteq \mathcal{M}^{p_0}_{q_0}$,
and
$\{h_j\}_{j=1}^\infty \subseteq \mathcal{M}^{p_1}_{q_1}$ such that $f=f_j+g_j+h_j$, 
\begin{align}\label{eq:171221-10}
\lim_{j\to \infty}
\|g_j\|_{\mathcal{M}^{p_0}_{q_0}}
=0
\ 
{\rm and}
\ 
\lim_{j\to \infty}
\|h_j\|_{\mathcal{M}^{p_1}_{q_1}}
=0.
\end{align}
Note that, by virtue of Lemma \ref{lem:171220-1} and the inequality
\[
\chi_{\{a\le |f|\le b\}}
\Theta(|f_j|)
\le |f_j|,
\]
we have
$
\chi_{\{a\le |f|\le b\}}
\Theta(|f_j|)
\in V_\infty\mathcal{M}^p_q.
$ 
Therefore, \eqref{eq:171221-8} is valid once we can show that 
\begin{align}\label{eq:171221-9}
\lim_{j\to \infty}
\|\chi_{\{a\le |f|\le b\}}
(\Theta(|f_j|)-\Theta(|f|)\|_{\mathcal{M}^p_q}
=0
\end{align} 
Since 
\[
|\Theta(t_1)-\Theta(t_2)|
\lesssim
\min(1, |t_1-t_2|),
\]
%%%typo t-2->t_2
for every $t_1, t_2\ge 0$, we have
\begin{align}\label{eq:171221-12}
\|\chi_{\{a\le |f|\le b\}}
(\Theta(|f_j|)-\Theta(|f|)\|_{\mathcal{M}^p_q}
&\lesssim
\|\chi_{\{a\le |f|\le b\}}
\min(1, |g_j|)\|_{\mathcal{M}^p_q}
\nonumber
\\
&\quad+
\|\min(1,|h_j|)\|_{\mathcal{M}^p_q}.
\end{align}
By using $q>q_1$, we have
\begin{align}\label{eq:171221-17}
\|\min(1,|h_j|)\|_{\mathcal{M}^p_q}
\le \||h_j|^{q_1/q}\|_{\mathcal{M}^p_q}
=
\|h_j\|_{\mathcal{M}^{p_1}_{q_1}}^{q_1/q}.
\end{align}
Meanwhile, by using the H\"older inequality, we get 
\begin{align}\label{eq:171221-18}
\|\chi_{\{a\le |f|\le b\}}
\min(1, |g_j|)\|_{\mathcal{M}^p_q}
&\le 
\|g_j\|_{\mathcal{M}^{p_0}_{q_0}}^{1-\theta}
\|\chi_{\{a\le |f|\le b\}}
\min(1, |g_j|)\|_{\mathcal{M}^{p_1}_{q_1}}^{\theta}
\nonumber
\\
&\lesssim
\|g_j\|_{\mathcal{M}^{p_0}_{q_0}}^{1-\theta}
\|f\|_{\mathcal{M}^{p}_{q}}^{\frac{\theta q}{q_1}}.
\end{align}
Combining \eqref{eq:171221-10}, 
\eqref{eq:171221-12}, \eqref{eq:171221-17}, 
and 
\eqref{eq:171221-18}, we obtain 
\eqref{eq:171221-9}. 
\end{proof}

We are now ready to prove Theorem \ref{thm:171213-2}. 
\begin{proof}[Proof of Theorem \ref{thm:171213-2}]
By virtue of Theorem \ref{thm:HS2-2} and Lemma \ref{lem:171223-1}, we have 
\[
\mathcal{M}^p_q
\subseteq
[V_0\mathcal{M}^{p_0}_{q_0}, V_0\mathcal{M}^{p_1}_{q_1}]^\theta. 
\]
Combining this, $[V_0\mathcal{M}^{p_0}_{q_0}, V_0\mathcal{M}^{p_1}_{q_1}]^\theta
\subseteq [\mathcal{M}^{p_0}_{q_0}, \mathcal{M}^{p_1}_{q_1}]^\theta$, and  Theorem \ref{thm:Lemarie}, we have
\eqref{eq:thm:171213-2-1}.
%%%typo \eqref{eq:thm:171213-2-2}->\eqref{eq:thm:171213-2-1}
Next, we only show \eqref{eq:thm:171213-2-2} because we can prove 
\eqref{eq:thm:171213-2-3} by a similar argument.  
Let $f\in [V_\infty\mathcal{M}^{p_0}_{q_0}, V_\infty\mathcal{M}^{p_1}_{q_1}]^\theta$. 
Then there exists $G\in \mathcal{G}(V_\infty\mathcal{M}^{p_0}_{q_0}, V_\infty\mathcal{M}^{p_1}_{q_1})$ such that 
\begin{align*} 
f=G'(\theta).
\end{align*}
Consequently, 
\begin{align}\label{eq:171221-15}
\lim_{k\to \infty}
\|f-H_k(\theta)\|_{\mathcal{M}^{p_0}_{q_0}+\mathcal{M}^{p_1}_{q_1}}
=0,
\end{align}
where $H_k$ is defined in Lemma \ref{lem:160417-3}. Combining \eqref{eq:171221-15} 
with 
the second part of
%%%This is just a matter of my taste. Allow me to do so.
Theorem \ref{thm:171213-1}, Lemma \ref{lem:160417-3} and 
$[V_\infty\mathcal{M}^{p_0}_{q_0}, V_\infty\mathcal{M}^{p_1}_{q_1}]^\theta \subseteq \mathcal{M}^p_q$, we have
\[
f\in \mathcal{M}^p_q 
\cap \overline{V_\infty\mathcal{M}^p_q}^{\mathcal{M}^{p_0}_{q_0}+\mathcal{M}^{p_1}_{q_1}}.
\]
Therefore, by virtue of Lemma \ref{lem:171220-2}, we have $f\in \mathcal{M}^p_q$ and $\chi_{\{a\le|f|\le b\}}f\in V_\infty\mathcal{M}^p_q$ for every $0<a<b<\infty$. 
We now show that 
\begin{align}\label{eq:171220-19}
\bigcap_{0<a<b<\infty}
\{f\in \mathcal{M}^p_q:
\chi_{\{a\le |f|\le b\}}\in V_\infty\mathcal{M}^p_q\}
\subseteq 
[V_\infty\mathcal{M}^{p_0}_{q_0},
V_\infty\mathcal{M}^{p_1}_{q_1}]^\theta.
\end{align}
Suppose that $f$ belongs to the set in the left-hand side of \eqref{eq:171220-19}. 
Let $F$ and $G$ be defined by \eqref{eq:171220-20} and \eqref{eq:171220-210}, respectively. In view of Proposition \ref{prop:Lemarie-HS}, we only need to show that
\begin{align}\label{eq:171220-23}
G(z)\in V_\infty\mathcal{M}^{p_0}_{q_0}
+V_\infty\mathcal{M}^{p_1}_{q_1} \quad (z\in \overline{S})
\end{align}
and 
\begin{align}\label{eq:171220-21}
G(k+it)-G(k)\in V_\infty\mathcal{M}^{p_k}_{q_k}
\quad 
(k\in \{0,1\}, t\in \mathbb{R}).
\end{align}
The proof of \eqref{eq:171220-23} goes as follows. 
Define $G_0(z):=\chi_{\{|f| \le 1\}}G(z)$ and 
$G_1(z):=G(z)-G_0(z)$. For every $\varepsilon\in (0,1)$, we write $G_\varepsilon(z):=\chi_{\{|f|\ge \varepsilon\}}G_0(z)$. Then, by \eqref{eq:171220-27}, we have
\[
|G_\varepsilon(z)|
\le 2(1+|z|)
\chi_{\{\varepsilon\le |f|\le 1\}}
\le 
\frac{2(1+|z|)}{\varepsilon}
\chi_{\{\varepsilon\le |f|\le 1\}}|f|.
\]
Since $\chi_{\{\varepsilon\le |f|\le 1\}}f\in V_\infty\mathcal{M}^{p_0}_{q_0}$, 
by virtue of Lemma \ref{lem:171220-1}, we have
$G_\varepsilon(z)\in V_\infty\mathcal{M}^{p_0}_{q_0}$. Meanwhile, 
\begin{align*}
|G_0(z)-G_\varepsilon(z)|
=
\left|
\chi_{\{|f|<\varepsilon\}}
\frac{F(z)-F(\theta)}{\left(\frac{p}{p_1}-\frac{p}{p_0}\right) \log |f|}
\right|
\lesssim
\frac{|f|^{\frac{p}{p_0}}}{-\log \varepsilon}.
\end{align*}
This implies
\[
\|G_0(z)-G_\varepsilon(z)\|_{\mathcal{M}^{p_0}_{q_0}}
\lesssim
\frac{\|f\|_{\mathcal{M}^p_q}^{\frac{p}{p_0}}}{-\log \varepsilon}.
\]
Therefore, $\lim\limits_{\varepsilon\to 0^{+}} \|G_0(z)-G_\varepsilon(z)\|_{\mathcal{M}^{p_0}_{q_0}}=0$. Consequently, $G_0(z)\in V_\infty\mathcal{M}^{p_0}_{q_0}$. By a similar argument, we also have $G_1(z)\in V_\infty\mathcal{M}^{p_1}_{q_1}$. 
Since $G(z)=G_0(z)+G_1(z)$, we obtain \eqref{eq:171220-23}. 
We now prove \eqref{eq:171220-21}. 
For every $N\in \mathbb{N}$, we define
\[
H_N(k+it)
:=(G(k+it)-G(k))\chi_{\{N^{-1}\le |f|\le N\}}.
\]
It follows from \eqref{eq:171220-27} that
\[
|H_N(k+it)|
\le C_{k, N, t} \chi_{\{N^{-1}\le |f|\le N\}}|f|.
\]
Therefore, by virtue of Lemma \ref{lem:171220-1}, we have 
$H_N(k+it)\in V_\infty\mathcal{M}^{p_k}_{q_k}$. 
Moreover, 
\begin{align*}
\|G(k+it)-G(k)-H_N(k+it)\|_{\mathcal{M}^{p_k}_{q_k}}
&\lesssim
\|
\chi_{\{|f|<N^{-1}\}\cup \{|f|>N\}}
\frac{|F(k+it)|+|F(k)|}{|\log |f||}
\|_{\mathcal{M}^{p_k}_{q_k}}
\\
&\lesssim
\frac{\left\||f|^{\frac{p}{p_k}}\right\|_{\mathcal{M}^{p_k}_{q_k}}}{\log N}
=
\frac{\|f\|_{\mathcal{M}^{p}_{q}}^{\frac{p}{p_k}}}{\log N}.
\end{align*}
Consequently, 
$
\lim\limits_{N\to \infty}
\|G(k+it)-G(k)-H_N(k+it)\|_{\mathcal{M}^{p_k}_{q_k}}=0$.
Combining this and $H_N(k+it)\in V_\infty\mathcal{M}^{p_k}_{q_k}$, we obtain \eqref{eq:171220-21}, as desired. 
\end{proof}

\section{Complex interpolation of $(\mathbb{M}^{p_0}_{q_0}, \mathbb{M}^{p_1}_{q_1})$}
\begin{theorem}\label{thm:171218-1}
Let $1 \le q \le p<\infty$. Then 
$\mathbb{M}^p_q=\overset{\diamond}{\mathcal M}{}^p_q$. 
\end{theorem}
\begin{proof}
Let $f \in \overset{\diamond}{\mathcal M}{}^p_q$.
Then there exists
$\{f_j\}_{j=1}^\infty \subset {\mathcal M}^p_q$
such that
$\partial^\alpha f_j\in {\mathcal M}^p_q$
for all 
$\alpha \in {\mathbb N}_0{}^n$ and $j \in {\mathbb N}$
and that
$\displaystyle
\lim_{j \to \infty}f_j=f
$
in the topology of ${\mathcal M}^p_q$.
Let $y \in {\mathbb R}^n$.
We observe
\begin{align*}
\|f(\cdot+y)-f\|_{{\mathcal M}^p_q}
&\le
\|f_j(\cdot+y)-f_j\|_{{\mathcal M}^p_q}
+
\|f(\cdot+y)-f_j(\cdot+y)\|_{{\mathcal M}^p_q}
+
\|f-f_j\|_{{\mathcal M}^p_q}\\
&\le
\|f_j(\cdot+y)-f_j\|_{{\mathcal M}^p_q}
+2
\|f-f_j\|_{{\mathcal M}^p_q}.
\end{align*}
We note that
each $f_j$ is smooth in view of the fact that
$f \in {\rm BUC}$
whenever
$\partial^\alpha f \in {\mathcal M}^p_q$
for all $\alpha$ such that $|\alpha| \le \dfrac{n}{p}+1$.
So,
by the mean value theorem,
\begin{align*}
\|f_j(\cdot+y)-f_j\|_{{\mathcal M}^p_q}
&=
\left\|
\int_0^1 y \cdot \nabla f_j(\cdot+t y)\,dt
\right\|_{{\mathcal M}^p_q}\\
&\le
\int_0^1 |y| \cdot \|\nabla f_j(\cdot+t y)\|_{{\mathcal M}^p_q}\,dt\\
&=|y| \cdot \|\nabla f_j\|_{{\mathcal M}^p_q}.
\end{align*}
%I just took care of the marginal typo.
Thus,
\begin{align*}
\|f(\cdot+y)-f\|_{{\mathcal M}^p_q}
&\le
|y| \cdot \|\nabla f_j\|_{{\mathcal M}^p_q}
+2
\|f-f_j\|_{{\mathcal M}^p_q}.
\end{align*}
If we let $y \to 0$,
then we obtain
\[\limsup_{y \to 0}
\|f(\cdot+y)-f\|_{{\mathcal M}^p_q}
\le
2\|f-f_j\|_{{\mathcal M}^p_q}.
\]
It remains to let $j \to \infty$.

Conversely let $f \in {\mathbb M}^p_q$.
Choose a non-negative function $\rho \in C^\infty(\{|y|<1\})$
with $\|\rho\|_{L^1}=1$.
Set $\rho_j=j^n \rho(j \cdot)$
for $j \in {\mathbb N}$.
We set
$f_j=\rho_j*f$.
Then we have
\[
\|f-f_j\|_{{\mathcal M}^p_q}
\le
\int_{{\mathbb R}^n}\rho_j(y)\|f-f(\cdot-y)\|_{{\mathcal M}^p_q}\,dy
\le
\sup_{|y| \le j^{-1}}\|f-f(\cdot-y)\|_{{\mathcal M}^p_q}.
\] 
As a result, letting $j \to \infty$,
we obtain
$\displaystyle
\lim_{j \to \infty}f_j=f
$
in the topology of ${\mathcal M}^p_q$.
Since $\partial^\alpha f_j=(\partial^\alpha \rho_j)*f \in {\mathcal M}^p_q$,
we obtain 
$f \in \overset{\diamond}{\mathcal M}{}^p_q$.
\end{proof}
As a corollary of Theorems \ref{thm:HNS} and \ref{thm:171218-1}, we have the following result.
\begin{corollary}
Keep the same assumption as in Theorem \ref{thm:HNS}. Then 
\begin{align*}
[{\mathbb M}^{p_0}_{q_0},
{\mathbb M}^{p_1}_{q_1}]_\theta
=
\{f\in
{\mathbb M}^{p}_{q}
: \lim_{N\to \infty}
\|f-\chi_{\{1/N\le |f|\le N\}}f\|_{\cM^p_q}=0
\}
\end{align*}
and
\begin{align*}
[{\mathbb M}^{p_0}_{q_0},
{\mathbb M}^{p_1}_{q_1}]^\theta
=
\bigcap_{0<a<1}
\left\{
f \in {\mathcal M}^p_q
\,:\,
\lim_{J \to \infty}
\|S(f;a,J)\|_{{\mathcal M}^p_q}
=0
\right\}.
\end{align*}

\end{corollary}

\section{Examples}
In this section, we shall examine the relation between each subspace in Theorems 
\ref{thm:171213-1} and \ref{thm:171213-2}  and comparing them by giving several examples. 
Let $\theta \in (0,1)$ and assume that 
\begin{align}\label{eq:171227-8}
1\le q_0<p_0<\infty, \ 1\le q_1<p_1<\infty, \ {\rm and} \  \frac{p_0}{q_0}=\frac{p_1}{q_1}.
\end{align}
Let $p$ and $q$ be defined by \eqref{eq:pq}.
Define 
$$
E(p,q):=\{(y_k+(R-1)(a_1+R a_2+\cdots))_{k=1, \ldots, n}\,:\,
\{a_j\}_{j=1}^\infty \in \{0,1\}^{\infty} \cap \ell^1({\mathbb N}), y\in [0,1]^n\},
$$
where $R>2$ solves $R^{\frac{n}{p}-\frac{n}{q}}2^{\frac{n}{q}}=1$,
so that each connected component of $E(p,q)$ is a closed cube with volume $1$.
It is known that $\chi_{E(p,q)} \in {\mathcal M}^p_q({\mathbb R}^n)$;
see \cite{SST11-1}.
We also define 
\[
E_m:=\bigcup_{j=0}^{m-1} [j, j+m^{-nq/p}]\times [0,m]^{n-1}
\]
for $m\in \mathbb{N}$.
For $x \in {\mathbb R}^n$,
we let
\begin{align*}
f_1(x)&:=\chi_{E(p,q)}(x)\\
f_2(x)&:=\sum_{j=1}^\infty \chi_{j!{\bf e}_1+[0,1]^n}(x)\\
f_3(x)&:=|x|^{-n/p}(x)\\
f_4(x)&:=\sum_{m=1}^\infty \chi_{E_m}(x-m!e_1).
\end{align*}
We have the following list of the membership:
\[
\begin{tabular}{|
@{\hspace{5pt}}c@{\hspace{5pt}}||%
@{\hspace{5pt}}c@{\hspace{5pt}}|%
@{\hspace{5pt}}c@{\hspace{5pt}}|%
@{\hspace{5pt}}c@{\hspace{5pt}}|%
@{\hspace{5pt}}c@{\hspace{5pt}}|%
}
\hline
 & $f_1$ & $f_2$ & $f_3$ & $f_4$\\ 
\hline
$[V_0{\mathcal M}^{p_0}_{q_0},V_0{\mathcal M}^{p_1}_{q_1}]_\theta=
[{\mathcal M}^{p_0}_{q_0},{\mathcal M}^{p_1}_{q_1}]_\theta$&$\circ$&$\circ$&$\times$&$\circ$
\\
\hline
$[V_\infty{\mathcal M}^{p_0}_{q_0},V_\infty{\mathcal M}^{p_1}_{q_1}]_\theta$&$\times$&$\circ$&$\times$&$\times$
\\ 
\hline
$[V^{(*)}{\mathcal M}^{p_0}_{q_0},V^{(*)}{\mathcal M}^{p_1}_{q_1}]_\theta$&$\times$&$\times$&$\times$&$\circ$\\
\hline
$[V_0{\mathcal M}^{p_0}_{q_0},V_0{\mathcal M}^{p_1}_{q_1}]^\theta={\mathcal M}^p_q$&$\circ$&$\circ$&$\circ$&$\circ$\\
\hline
$[V_\infty{\mathcal M}^{p_0}_{q_0},V_\infty{\mathcal M}^{p_1}_{q_1}]^\theta$&$\times$&$\circ$&$\circ$&$\times$\\ 
\hline
$[V^{(*)}{\mathcal M}^{p_0}_{q_0},V^{(*)}{\mathcal M}^{p_1}_{q_1}]^\theta$&$\times$&$\times$&$\circ$&$\circ$\\
\hline
\end{tabular}
\]
In the table,
$\circ$ stands for the membership,
while
$\times$ means that the function does not belong to the function space. The detail verification of this table is given as follows. 
\begin{corollary}\label{cor:171226-1}
Let $\theta \in (0,1)$ and assume \eqref{eq:171227-8}. 
Then, 
$f_1$ belongs to 
$[V_0\mathcal{M}^{p_0}_{q_0}, V_0\mathcal{M}^{p_1}_{q_1}]_\theta$, but 
$$f_1 \notin [V^{(*)}\mathcal{M}^{p_0}_{q_0}, V^{(*)}\mathcal{M}^{p_1}_{q_1}]_\theta \cup [V_\infty\mathcal{M}^{p_0}_{q_0}, V_\infty\mathcal{M}^{p_1}_{q_1}]_\theta.$$
\end{corollary}
\begin{proof}
Observe that, for every $N\in \mathbb{N}$, we have
\[
f_1-\chi_{\{\frac1N\le |f_1|\le N\}} f_1=f_1\chi_{\{|f_1|\le \frac1N\}\cup \{|f_1|>N\}}=0.
\]
Therefore, 
\begin{align}\label{eq:171226-1}
\lim_{N\to \infty}
\|f_1-\chi_{\{\frac1N\le |f_1|\le N\}} f_1\|_{\mathcal{M}^p_q}=0.
\end{align}
Since $f_1\in \mathcal{M}^p_q$ satisfies \eqref{eq:171226-1}, by virtue of Theorem \ref{thm:171213-1}, we have  $f_1 \in [V_0\mathcal{M}^{p_0}_{q_0}, V_0\mathcal{M}^{p_1}_{q_1}]_\theta$. According to Theorem \ref{thm:171213-1}, we only need to show that 
\begin{align}\label{eq:171226-19}
f_1\notin V^{(*)}\mathcal{M}^p_q \cup V_\infty\mathcal{M}^p_q. 
\end{align}
Let $N\in \mathbb{N}$. Then there exists a closed cube $Q=Q_N\subseteq E(p,q)$ of length $1$ such that $Q\subseteq \mathbb{R}^n \setminus B(0, 2N)$. Since 
\begin{align*}
\sup_{x\in \mathbb{R}^n}
\int_{B(x,1)} |f_1(y)|^q \chi_{\mathbb{R}^n\setminus B(0,N)}(y) \ dy 
&\ge \int_{B(c_Q, 1)} \chi_Q(y) \chi_{\mathbb{R}^n \setminus B(0, N)}(y) \ dy
\\
&=|Q \cap B(c_Q, 1)|
\ge  \left(\frac{2}{\sqrt{n}}\right)^n,
\end{align*}
we see that 
\[
\lim_{N\to \infty}
\sup_{x\in \mathbb{R}^n}
\int_{B(x,1)} |f_1(y)|^q \chi_{\mathbb{R}^n\setminus B(0,N)}(y) \ dy 
\neq 0.
\]
Hence, $f_1\notin V^{(*)} \mathcal{M}^p_q$. 
We now show that $f_1\notin V_\infty\mathcal{M}^p_q$. 
Let $k\in \mathbb{N}$. By a geometric observation, $[0, R^k]^n \cap E(p,q)=\cup_{j=1}^{2^{kn}} Q_j$, where $\{Q_j\}_{j=1}^{2^{kn}}$ is a collection of closed cube of length 1. Therefore,
\begin{align*}
m\left(f_1, p, q; \frac{\sqrt{n}R^k}{2}\right)
&\gtrsim 
|B(c_{[0, R^k]^n}, R^k)|^{\frac1p-\frac1q}
\left(
\int_{[0, R^k]^n} |f_1(y)|^q \ dy\right)^{\frac1q}
\\
&\sim 
R^{k\left(\frac{n}{p}-\frac{n}{q}\right)}
|[0, R^k]^n \cap E(p,q)|^{\frac1q} 
= 
R^{k\left(\frac{n}{p}-\frac{n}{q}\right)}
2^{\frac{kn}{q}}=1.
\end{align*}
Since $k$ is arbitrary, we see that 
$\lim\limits_{r\to \infty} m(f_1, p, q; r)\ne 0$, so $f_1\notin V_\infty\mathcal{M}^p_q$, as desired. 
\end{proof}

\begin{corollary}\label{cor:171226-2}
Keep the same assumption as in Corollary \ref{cor:171226-1}. 
Then, $f_1$ belongs to 
$[V_0\mathcal{M}^{p_0}_{q_0}, V_0\mathcal{M}^{p_1}_{q_1}]^\theta$, but 
$$f_1\notin [V^{(*)}\mathcal{M}^{p_0}_{q_0}, V^{(*)}\mathcal{M}^{p_1}_{q_1}]^\theta \cup [V_\infty\mathcal{M}^{p_0}_{q_0}, V_\infty\mathcal{M}^{p_1}_{q_1}]^\theta.$$
\end{corollary}
\begin{proof}
The first assertion follows from Corollary \ref{cor:171226-1} and 
\[
[V_0\mathcal{M}^{p_0}_{q_0}, V_0\mathcal{M}^{p_1}_{q_1}]_\theta \subseteq [V_0\mathcal{M}^{p_0}_{q_0}, V_0\mathcal{M}^{p_1}_{q_1}]^\theta.
\] 
Combining \eqref{eq:171226-19}, Theorem \ref{thm:171213-2}, 
and the identity 
\[
\chi_{\{1/2\le |f_1|\le 1\}} f_1=f_1,
\]
we conclude that $f_1\notin [V^{(*)}\mathcal{M}^{p_0}_{q_0}, V^{(*)}\mathcal{M}^{p_1}_{q_1}]^\theta \cup [V_\infty\mathcal{M}^{p_0}_{q_0}, V_\infty\mathcal{M}^{p_1}_{q_1}]^\theta.$
\end{proof}

\begin{corollary}\label{cor:171226-3}
Keep the same assumption as in Corollary \ref{cor:171226-1}. 
Then,
$f_2$ belongs to 
$[V_\infty \mathcal{M}^{p_0}_{q_0}, V_\infty\mathcal{M}^{p_1}_{q_1}]_\theta$, but 
$f_2\notin [V^{(*)}\mathcal{M}^{p_0}_{q_0}, V^{(*)}\mathcal{M}^{p_1}_{q_1}]_\theta$. 
\end{corollary}
\begin{proof}
By a similar argument as in the proof of  \cite[Theorem 4.1]{AS}, we have $f_2\in V_\infty\mathcal{M}^p_q$ but $f_2 \notin V^{(*)} \mathcal{M}^p_q$. 
Moreover, $f_2$ satisfies 
\[
\lim_{N\to \infty}
\|f_2-f_2\chi_{\{1/N\le |f_2|\le N\}}\|_{\mathcal{M}^p_q}=0.
\]
Therefore, by virtue of Theorem \ref{thm:171213-1},
we have the desired conclusion. 
\end{proof}
\begin{corollary}\label{cor:171226-4}
Keep the same assumption as in Corollary \ref{cor:171226-1}. 
Then,
$f_2$ belongs to 
$[V_\infty \mathcal{M}^{p_0}_{q_0}, V_\infty\mathcal{M}^{p_1}_{q_1}]^\theta$, but 
$f_2\notin [V^{(*)}\mathcal{M}^{p_0}_{q_0}, V^{(*)}\mathcal{M}^{p_1}_{q_1}]^\theta$. 
\end{corollary}
\begin{proof}
Note that, $f_2 \in [V_\infty \mathcal{M}^{p_0}_{q_0}, V_\infty\mathcal{M}^{p_1}_{q_1}]^\theta$ is a consequence 
of Corolary \ref{cor:171226-3} and 
\[
[V_\infty \mathcal{M}^{p_0}_{q_0}, V_\infty\mathcal{M}^{p_1}_{q_1}]_\theta
\subseteq 
[V_\infty \mathcal{M}^{p_0}_{q_0}, V_\infty\mathcal{M}^{p_1}_{q_1}]^\theta.
\]
Meanwhile, by the identity 
\[
\chi_{\{1/2\le |f_2|\le 1\}}f_2=f_2
\] 
and  $f_2 \in V^{(*)}\mathcal{M}^p_q$, we have 
$f_2 \notin [V^{(*)}\mathcal{M}^{p_0}_{q_0}, V^{(*)}\mathcal{M}^{p_1}_{q_1}]^\theta$. 
\end{proof}

\begin{corollary}\label{cor:171226-5}
Keep the same assumption as in Corollary 
\ref{cor:171226-1}. Then
\begin{align}\label{eq:171226-10}
f_3 \in [V_0\mathcal{M}^{p_0}_{q_0}, V_0\mathcal{M}^{p_1}_{q_1}]^\theta
\setminus [V_0\mathcal{M}^{p_0}_{q_0}, V_0\mathcal{M}^{p_1}_{q_1}]_\theta,
\end{align}
\[
f_3 \in [V^{(*)}\mathcal{M}^{p_0}_{q_0}, V^{(*)}\mathcal{M}^{p_1}_{q_1}]^\theta
\cap [V_\infty\mathcal{M}^{p_0}_{q_0}, V_\infty,\mathcal{M}^{p_1}_{q_1}]^\theta
\]
and
\begin{align}\label{eq:171226-5}
f_3 \notin [V^{(*)}\mathcal{M}^{p_0}_{q_0}, V^{(*)}\mathcal{M}^{p_1}_{q_1}]_\theta
\cup [V_\infty\mathcal{M}^{p_0}_{q_0}, V_\infty\mathcal{M}^{p_1}_{q_1}]_\theta.
\end{align}
\end{corollary}
\begin{proof}
Since $f_3\in \mathcal{M}^p_q$, by virtue of Theorem \ref{thm:171213-2}, we have $f_3 \in [V_0\mathcal{M}^{p_0}_{q_0}, V_0\mathcal{M}^{p_1}_{q_1}]^\theta$.
Note that $f_3$ fails to belong
to $[V_0{\mathcal M}^{p_0}_{q_0},V_0{\mathcal M}^{p_1}_{q_1}]_\theta$
because
\[
f_3=\lim_{N \to \infty}\chi_{[N^{-1},N]}(f)f
\]
fails in ${\mathcal M}^p_q$.

Let $0<a<b<\infty$. Since $\chi_{\{a\le |f_3|\le b\}}f_3 \in L^\infty_{\rm c}$, by virtue of Lemma \ref{lem:171227-1}, we have $\chi_{\{a\le |f_3|\le b\}}f_3\in V^{(*)}\mathcal{M}^p_q \cap V_\infty \mathcal{M}^p_q$. Therefore, according to Theorem \ref{thm:171213-2}, we have
$f_3 \in [V^{(*)}\mathcal{M}^{p_0}_{q_0}, V^{(*)}\mathcal{M}^{p_1}_{q_1}]^\theta \cap [V_\infty\mathcal{M}^{p_0}_{q_0}, V_\infty\mathcal{M}^{p_1}_{q_1}]^\theta$. 
Meanwhile, \eqref{eq:171226-5} follows immediately from 
\[
([V^{(*)}\mathcal{M}^{p_0}_{q_0}, V^{(*)}\mathcal{M}^{p_1}_{q_1}]_\theta
\cup [V_\infty\mathcal{M}^{p_0}_{q_0}, V_\infty\mathcal{M}^{p_1}_{q_1}]_\theta) 
\subseteq 
[V_0\mathcal{M}^{p_0}_{q_0}, V_0\mathcal{M}^{p_1}_{q_1}]_\theta
\]
and  \eqref{eq:171226-10}.
\end{proof}
\begin{corollary}\label{cor:171226-6}
Keep the same assumption as in Corollary \ref{cor:171226-1}. 
Then,
$f_4$ belongs to 
$[V^{(*)} \mathcal{M}^{p_0}_{q_0}, V^{(*)}\mathcal{M}^{p_1}_{q_1}]^\theta$, but 
$f_4\notin [V_\infty\mathcal{M}^{p_0}_{q_0}, V_\infty \mathcal{M}^{p_1}_{q_1}]^\theta$. 
\end{corollary}
\begin{proof}
Let $x\in \mathbb{R}^n$ and $N\in \mathbb{N}$. 
Note that, if $m\in \mathbb{N}$ satisfies   
\[
(m!+m)^2+(n-1)m^2<N,
\]
then for every 
$y \in [m!, m!+m]\times [0, m]^{n-1}$, we have
$|y|<N$. Consequently, 
\begin{align}\label{eq:171227-7}
\int_{B(x,1)} &|f_4(y)|^q \chi_{\mathbb{R}^n\setminus B(0, N)}(y) \ dy 
\nonumber
\\
&=
\sum_{m\in \mathbb{N}}
\int_{B(x,1)} \chi_{E_m}(y-m!e_1) \chi_{\mathbb{R}^n\setminus B(0, N)}(y) \ dy 
\nonumber
\\
&=
\sum_{m\in \mathbb{N}, (m!+m)^2+(n-1)m^2\ge N}
\int_{B(x,1)} \chi_{E_m}(y-m!e_1) \chi_{\mathbb{R}^n\setminus B(0, N)}(y) \ dy 
\nonumber
\\
&\le 
\sum_{m\in \mathbb{N}, (m!+m)^2+(n-1)m^2\ge N}
\int_{B(x,1)} \chi_{E_m}(y-m!e_1) \ dy.
\end{align}
By a geometric observation, we see that
\begin{align}\label{eq:171227-17}
\int_{B(x,1)} \chi_{E_m}(y-m!e_1) \ dy.
&=
\int_{\mathbb{R}^n}
\chi_{E_m}(y) \chi_{B(x,1)}(y+m!e_1) \ dy
\nonumber
\\
&=
\int_{\mathbb{R}^n}
\chi_{E_m}(y) \chi_{B(x-m!e_1,1)}(y) \ dy
\nonumber
\\
&=
|E_m \cap B(x-m!e_1,1)| 
\nonumber
\\
&\lesssim 
\frac{|E_m|}{m^n}
=
\frac{m^{-\frac{nq}{p}} \cdot m^{n-1}}{m^n}
=m^{-\frac{nq}{p}-1}.
\end{align}
Combining \eqref{eq:171227-7} and \eqref{eq:171227-17}, we get 
\[
\sup_{x\in \mathbb{R}^n}
\int_{B(x,1)} |f_4(y)|^q \chi_{\mathbb{R}^n\setminus B(0, N)}(y) \ dy 
\lesssim
\sum_{m\in \mathbb{N}, (m!+m)^2+(n-1)m^2\ge N}
m^{-\frac{nq}{p}-1}.
\]
Since $\sum\limits_{m\in \mathbb{N}}
m^{-\frac{nq}{p}-1}<\infty$, we have
\[
\lim_{N\to \infty}
\sup_{x\in \mathbb{R}^n}
\int_{B(x,1)} |f_4(y)|^q \chi_{\mathbb{R}^n\setminus B(0, N)}(y) \ dy 
=0,
\]
so $f_4\in V^{(*)}\mathcal{M}^p_q$. According to Remark \ref{rem:171220-1}, we have
\[
\chi_{\{a\le |f_4|\le b\}}f_4 \in V^{(*)}\mathcal{M}^p_q,
\]
for every $0<a<b<\infty$. 
Therefore, by virtue of \eqref{eq:thm:171213-2-3}, we conclude that 
$$f_4 \in [V^{(*)} \mathcal{M}^{p_0}_{q_0}, V^{(*)}\mathcal{M}^{p_1}_{q_1}]^\theta.$$ 
We now prove that 
$f_4 \notin [V_\infty \mathcal{M}^{p_0}_{q_0}, V_\infty\mathcal{M}^{p_1}_{q_1}]^\theta$. 
For every $N\in \mathbb{N}$, we define 
$$Q:=[N!, N!+N]\times [0, N]^{n-1}.$$ 
Since $|E_N|=N^{n-\frac{nq}{p}}$, we have
\begin{align*}
m(f_4, p, q;\sqrt{n}N/2)
&\gtrsim 
N^{\frac{n}{p}-\frac{n}{q}}
\left(
\int_{B(c_Q, \sqrt{n}N/2)}
|f_4(y)|^q  \ dy
\right)^{\frac1q}
\\
&\ge 
N^{\frac{n}{p}-\frac{n}{q}}
\left(
\int_{Q}
\chi_{E_N}(y-N!e_1)  \ dy
\right)^{\frac1q}
\\
&=
N^{\frac{n}{p}-\frac{n}{q}}
|E_N|^{\frac1q}=1,
\end{align*}
so $\lim\limits_{r\to \infty} m(f_4, p, q;r)\neq 0$. 
Therefore, $f_4 \notin V_\infty\mathcal{M}^p_q$. 
Combining this with \eqref{eq:thm:171213-2-2} and
\[
\chi_{\{1/2\le |f_4|\le 1\}}f_4=f_4,
\]
we conclude that $f_4 \notin [V_\infty \mathcal{M}^{p_0}_{q_0}, V_\infty\mathcal{M}^{p_1}_{q_1}]^\theta$. 
\end{proof}

\begin{corollary}\label{cor:171226-7}
Keep the same assumption as in Corollary \ref{cor:171226-1}. 
Then,
$f_4$ belongs to 
$[V^{(*)} \mathcal{M}^{p_0}_{q_0}, V^{(*)}\mathcal{M}^{p_1}_{q_1}]_\theta$, but 
$f_4\notin [V_\infty\mathcal{M}^{p_0}_{q_0}, V_\infty \mathcal{M}^{p_1}_{q_1}]_\theta$. 
\end{corollary}
\begin{proof}
In the proof of Corollary \ref{cor:171226-6}, it is shown that $f_4 \in V^{(*)}\mathcal{M}^p_q$. Moreover, for every $N\in \mathbb{N}$, we have
\[
\chi_{\{|f_4|<1/N\}\cup \{|f_4|> N\}}f_4=0,
\]
so $\lim\limits_{N\to \infty} \|f_4-\chi_{\{1/N\le |f_4|\le N\}}f_4\|_{\mathcal{M}^p_q}=0$. Therefore, by \eqref{eq:171213-4}, we have $$f_4 \in [V^{(*)} \mathcal{M}^{p_0}_{q_0}, V^{(*)}\mathcal{M}^{p_1}_{q_1}]_\theta.$$
The second assertion follows from  
\[
[V_\infty\mathcal{M}^{p_0}_{q_0}, V_\infty \mathcal{M}^{p_1}_{q_1}]_\theta
\subseteq 
[V_\infty\mathcal{M}^{p_0}_{q_0}, V_\infty \mathcal{M}^{p_1}_{q_1}]^\theta
\]
and Corollary \ref{cor:171226-6}.
\end{proof}

\end{document}